\documentclass[12pt]{article}

\usepackage{amssymb,amsmath,amsfonts}
\usepackage{amsthm}
\usepackage{enumitem}
\usepackage{color}
\usepackage{tikz}
\usepackage{booktabs,threeparttable}
\usepackage{color}

\numberwithin{equation}{section} 
\theoremstyle{plain}
\newtheorem{theorem}{Theorem}[section]

\newtheorem{corollary}[theorem]{Corollary}

\newtheorem{lemma}[theorem]{Lemma}

\theoremstyle{definition}
\newtheorem{definition}{Definition}[section]

\theoremstyle{remark}

\newcommand{\qbinom}[3]{\left[\genfrac{}{}{0pt}{}{#1}{#2}\right]_{#3}}

\allowdisplaybreaks

\begin{document}

\title{\bf Exploring $q$-Stancu Operators via a New Representation}
\author{Feride Baraner$^1$, Ovgu Gurel$^{1,}\thanks{Corresponding Author}$}
\date{}
\maketitle

\begin{center}
{\it $^1$Recep Tayyip Erdogan University, Department of Mathematics, 53100, Rize, Turkey}\\
{\it e-mail: feride$\_$baraner$23$@erdogan.edu.tr, ovgu.gurelyilmaz@erdogan.edu.tr}\\
\end{center}

\begin{abstract}
This paper investigates the $q$-Stancu operators, which generalize the $q$-Bernstein operators, by developing a new representation in terms of the $q$-Pochhammer symbol. Based on this representation, some known properties are re-discovered, and a general recurrence relation for the moments is established. It is shown that higher-order moments can be expressed in terms of lower-order ones. Furthermore, the limit form of the operators is defined and their uniform convergence is proved. Finally, the moments of the limit operator and their recurrence relations are presented.
 
\end{abstract}

{\bf Keywords}: $q$-Bernstein operator, $q$-Stancu operator, uniform convergence, limit $q$-Stancu operator

{\bf 2020 MSC:} 41A10, 41A36, 47B38.


\section{Introduction}

Quantum calculus has found a wide variety of applications in mathematical physics, statistics, Lie algebras, and combinatorial analysis \cite{kac}. One particularly fruitful direction has been the application of $q$-calculus to the theory of linear positive operators, which naturally led to $q$-generalizations of classical approximation operators. Among these, $q$-Bernstein polynomials -pioneered independently by Lupaş \cite{lupas} and Phillips \cite{phillips}- have been extensively studied in the literature \cite{oruc, ostrovska2003, ostrovska2006, ostrovska2014, turan, videnskii}.

Motivated by these developments, Nowak \cite{nowak} introduced the $q$-analogue of the Stancu operators in 2009, obtained a $q$-difference representation of these operators, and studied their approximation properties. Various features of these operators have been investigated from different perspectives ever since. Their approximation properties, convergence rates, and Voronovskaya-type theorems have been examined using tools such as the modulus of continuity, probabilistic approaches, and $q$-calculus techniques \cite{agratini2010,jiang,nowak_gupta2011}. Iteration properties, eigenvalues, eigenvectors, and shape-preserving properties have also been examined \cite{wang2013,wang2014}. Furthermore, alternative representations based on $q$-Beta functions have been proposed, together with corresponding Korovkin-type theorems and limit operators \cite{finta2016a}. Several generalizations, including Durrmeyer and Kantorovich variants, have been considered in subsequent studies \cite{erencin2012, gupta2016, neer2017}.

The study of limit forms of $q$-Bernstein-type operators has played an important role in the literature \cite{Ilinskii, mahmudov, ostrovska2023}. However, obtaining the limit of the $q$-Stancu operators introduced by Nowak has not been possible until now, due to the divergence of certain infinite products appearing in the basis polynomials. The new representation developed in this paper resolves this issue and allows for an explicit definition of the limit $q$-Stancu operator.

For the convenience of the reader, we now collect the basic notations and results from $q$-calculus that will be used throughout the paper. 

Let $q>0$. For any $n \in \mathbb{N}_0$, the $q$-integer is defined by
\begin{align}
[n]_q := 1+q+\cdots+q^{n-1}, \quad [0]_q := 0, \label{qint}
\end{align}
which reduces to the classical integer when $q=1$. The corresponding $q$-factorial is given by
\[
[n]_q! := [1]_q[2]_q\cdots[n]_q, \quad [0]_q! := 1.
\]

Moreover, for $c \in \mathbb{C}$, the $q$-Pochhammer symbol is defined by
\begin{align} 
(a;q)_0 := 1,\qquad
(a;q)_m := \prod_{j=0}^{m-1}(1-aq^j),\; m \in \mathbb{N},\qquad
(a;q)_\infty := \prod_{j=0}^{\infty}(1-aq^j).  \label{qpoch}
\end{align} 

Using these notions, for $k,n \in \mathbb{N}_0$ with $0 \le k \le n$, the $q$-binomial coefficient is defined as
\begin{align}
\qbinom{n}{k}{q} := \frac{[n]_q!}{[k]_q![n-k]_q!}= \frac{(q;q)_n}{(q;q)_k \, (q;q)_{n-k}}. \label{binompoch}
\end{align}
It is well-known that, for $q=1$, this coefficient coincides with the classical binomial coefficient.

The following identities from $q$-calculus \cite{askey} will be needed in the sequel,
\begin{align}
(ab; q)_n = \sum_{k=0}^{n} \begin{bmatrix} n \\ k \end{bmatrix}_q b^k (a; q)_k (b; q)_{n-k}, \quad 0 \le k \le n, \label{30}
\end{align}
and, for $k \in \mathbb{N}_0$, $|x| < 1$, $|q| < 1$,
\begin{align}
\sum_{k=0}^{\infty} \frac{(a; q)_k}{(q; q)_k} x^k = \frac{(ax; q)_{\infty}}{(x; q)_{\infty}}. \label{31}
\end{align}

The $q$-analogue of the Stancu operators was introduced by Nowak \cite{nowak} as follows.
\begin{definition}
Let $f \in C[0,1]$, $q \in (0,1)$, $\alpha>0$. The $n$-th $q$-Stancu polynomial is defined by
\begin{align}
\left(S_{n}^{q,\alpha}f\right)(x)= \sum_{k=0}^{n} f\left(\frac{[k]_q}{[n]_q}\right)p^{q,\alpha}_{n,k}(x), \quad n=1,2,\dots, \label{qstancu}
\end{align}
where $p^{q,\alpha}_{n,k}(x)$ are the $q$-Stancu basis polynomials given by
\begin{align}
p^{q,\alpha}_{n,k}(x)= \qbinom{n}{k}{q} \frac{\prod_{i=0}^{k-1}(x+\alpha[i]_q) \prod_{s=0}^{n-1-k}(1-q^s x+\alpha[s]_q)}{\prod_{i=0}^{n-1}(1+\alpha[i]_q)}, \quad k=0,1,\dots,n. \label{oldbasis}
\end{align}
\end{definition}

The paper is organized as follows. In Section 2, we present a new representation of the $q$-Stancu operators introduced by Nowak, obtained via the $q$-Pochhammer symbol. Using this representation, we provide alternative proofs of the first three moments of the operators, differing from the approach in Nowak's original work. Moreover, recurrence relations for the $m$-th order moments are established, showing that these moments can be expressed in terms of lower-order moments. In Section 3, we define the limit case of the operators and investigate its basic properties. We note that the original definition given by Nowak does not allow for the computation of the limit operator, due to the divergence of a product appearing in the structure of the operators. The new representation introduced in this paper enables an explicit definition of the limit operator.

\section{New Representation and Basic Properties}

In this section, the coefficients forming the basis polynomials of the $q$-Stancu operators introduced by Nowak in \cite{nowak} are rewritten in terms of the $q$-Pochhammer symbol given in \eqref{qpoch}, obtaining an alternative representation of the operators. The following lemma provides a key step toward this alternative representation. In the sequel, all results are obtained under the assumption $0<q<1$.\\
\begin{lemma} \label{pochrep}
Let $n\in\mathbb{N}$ and $\alpha\ge0$, and define $\gamma=\alpha/(1-q)$. 
Then, for $0\le k\le n$, the following identities hold:
\begin{enumerate}[label=(\roman*)]
\setlength{\itemindent}{3em}
	\item [$(i)$] $\displaystyle \prod_{i=0}^{k-1}\bigl(x+\alpha[i]_q\bigr)=(x+\gamma)^k\left(\frac{\gamma}{x+\gamma};q\right)_k$, 
	\item [$(ii)$] $ \displaystyle \prod_{i=0}^{n-k-1}\bigl(1-q^i x+\alpha[i]_q\bigr)=(1+\gamma)^{n-k}\left(\frac{x+\gamma}{1+\gamma};q\right)_{n-k}$.
\end{enumerate}
\end{lemma}

\begin{proof}
\textit{(i)} From definition \eqref{qint}, we have
\[
\prod_{i=0}^{k-1}\bigl(x+\alpha[i]_q\bigr)
=\prod_{i=0}^{k-1}\left(x+\alpha\frac{1-q^i}{1-q}\right)
=\prod_{i=0}^{k-1}\bigl(x+\gamma(1-q^i)\bigr).
\]
Factoring out $(x+\gamma)$ from each term, we obtain
\[
\prod_{i=0}^{k-1}\bigl(x+\gamma-\gamma q^i\bigr)
=(x+\gamma)^k
\prod_{i=0}^{k-1}\left(1-\frac{\gamma}{x+\gamma}q^i\right).
\]
Using the $q$-Pochhammer symbol defined in \eqref{qpoch}, this yields
\[
\prod_{i=0}^{k-1}\bigl(x+\alpha[i]_q\bigr)
=(x+\gamma)^k\left(\frac{\gamma}{x+\gamma};q\right)_k.
\]

\textit{(ii)} Using \eqref{qint} together with \eqref{qpoch}, following the same steps as in part \textit{(i)}, one has
\begin{align*}
\prod_{i=0}^{n-k-1}\bigl(1-q^i x+\alpha[i]_q\bigr)
&=
\prod_{i=0}^{n-k-1}\bigl(1+\gamma-(x+\gamma)q^i\bigr) \\
&=(1+\gamma)^{\,n-k}
\prod_{i=0}^{n-k-1}
\left(1-\frac{x+\gamma}{1+\gamma}q^i\right) \\
&=
(1+\gamma)^{\,n-k}
\left(\frac{x+\gamma}{1+\gamma};q\right)_{n-k}.
\end{align*}
\end{proof}
Based on the above results, we now establish an alternative representation of the $q$-Stancu operators. Taking into account identities in Lemma \ref{pochrep}, the basis polynomials \eqref{oldbasis} can be rewritten in the following form.
\begin{align}
p_{n,k}^{q,\alpha}(x)
&=
\qbinom{n}{k}{q}
\frac{
\left(\dfrac{\gamma}{x+\gamma};q\right)_k
\left(\dfrac{x+\gamma}{1+\gamma};q\right)_{n-k}
}{
\left(\dfrac{\gamma}{1+\gamma};q\right)_n
}
\left(\dfrac{x+\gamma}{1+\gamma}\right)^k,
\qquad k=0,1,\ldots,n.
\label{newbasis}
\end{align}
This new form of the basis polynomials will serve as the main tool in the rest of the paper. In particular, it allows us to compute the moments of the $q$-Stancu operators via a different proof technique. The first three moments were obtained by Nowak in 2009 using the $q$-difference representation given in \cite{nowak}. Here, we present an alternative derivation based on this new form.

\begin{theorem} \label{moment}
Let $\alpha \ge 0$. Then, for every $n \in \mathbb{N}$ and $x \in [0,1]$, we have
\begin{align*}
S_n^{q,\alpha}(e_0;x) &= 1,\\
S_n^{q,\alpha}(e_1;x) &= x,\\
S_n^{q,\alpha}(e_2;x) &=
\frac{1}{1+\alpha}
\left\{x(x+\alpha)+\frac{x(1-x)}{[n]_q}\right\}.
\end{align*}
\end{theorem}

\begin{proof}
From the definition of the $q$-Stancu operators with the basis polynomials represented in \eqref{newbasis}, one can write
\begin{align*}
S_n^{q,\alpha}(e_0;x)
&=
\frac{1}{\left(\dfrac{\gamma}{1+\gamma};q\right)_n}
\sum_{k=0}^{n}
\qbinom{n}{k}{q}
\left(\frac{x+\gamma}{1+\gamma}\right)^k
\left(\frac{\gamma}{x+\gamma};q\right)_k
\left(\frac{x+\gamma}{1+\gamma};q\right)_{n-k}.
\end{align*}
Taking $a=\gamma/(x+\gamma)$ and $b=(x+\gamma)/(1+\gamma)$ in identity \eqref{30}, we obtain
\[
S_n^{q,\alpha}(e_0;x)
=
\frac{1}{\left(\dfrac{\gamma}{1+\gamma};q\right)_n}
\left(\frac{\gamma}{1+\gamma};q\right)_n
=1.
\]

Now, let us take $f(t) = e_1(t)$. Then, one has
\begin{align*}
S_n^{q,\alpha}(e_1; x) &= \frac{1}{\left(\dfrac{\gamma}{1+\gamma};q\right)_n}
\sum_{k=0}^{n} 
\frac{[k]_q}{[n]_q} \begin{bmatrix} n \\ k \end{bmatrix}_q
\left(\frac{x+\gamma}{1+\gamma}\right)^k 
\left(\frac{\gamma}{x+\gamma};q\right)_k
\left(\frac{x+\gamma}{1+\gamma}; q\right)_{n-k} \\
&= \frac{1}{\left(\dfrac{\gamma}{1+\gamma};q\right)_n} 
\sum_{k=0}^{n-1} 
\begin{bmatrix} n-1 \\ k \end{bmatrix}_q
\left(\frac{x+\gamma}{1+\gamma}\right)^{k+1} 
\left(\frac{\gamma}{x+\gamma};q\right)_{k+1}  
\left(\frac{x+\gamma}{1+\gamma}; q\right)_{n-k-1}.
\end{align*}
Using the property of the $q$-Pochhammer symbol, we can write
\begin{align*}
S_n^{q,\alpha}(e_1; x) &= 
\frac{x+\gamma}{1+\gamma} \left( 1 - \frac{\gamma}{\gamma x + \gamma} \right)
\frac{1}{\left(\dfrac{\gamma}{1+\gamma};q\right)_n} \\
&\quad \times \sum_{k=0}^{n-1} 
\begin{bmatrix} n-1 \\ k \end{bmatrix}_q
\left(\frac{x+\gamma}{1+\gamma}\right)^k
\left(\frac{q\gamma}{x+\gamma}; q\right)_k
\left(\frac{x+\gamma}{1+\gamma}; q\right)_{n-k-1}.
\end{align*}
Taking into account expression \eqref{30}, we finally obtain
\[
S_n^{q,\alpha}(e_1; x) = \frac{x}{1+\gamma} \dfrac{\left(\dfrac{q\gamma}{x+\gamma}; q\right)_{n-1}}{\left(\dfrac{q\gamma}{x+\gamma}; q\right)_n} = x.
\]

For $f(t)=e_2(t)$, one has
\begin{align*}
S_{n}^{q,\alpha}(e_2; x)
&= \frac{1}{\left(\dfrac{\gamma}{1+\gamma}; q\right)_n}
\sum_{k=0}^{n}
\left(\frac{[k]_q}{[n]_q}\right)^2
\begin{bmatrix} n \\ k \end{bmatrix}_q
\left(\frac{x+\gamma}{1+\gamma}\right)^k
\left(\frac{\gamma}{x+\gamma}; q\right)_k
\left(\frac{x+\gamma}{1+\gamma}; q\right)_{n-k} \\
&= \frac{1}{\left(\dfrac{\gamma}{1+\gamma};\, q\right)_{\!n}}
   \sum_{k=0}^{n-1}
   \frac{[k+1]_q}{[n]_q}
   \begin{bmatrix} n-1 \\ k \end{bmatrix}_{\!q}
   \left(\frac{x+\gamma}{1+\gamma}\right)^{\!k+1}
   \left(\frac{\gamma}{x+\gamma};\, q\right)_{\!k+1}
   \left(\frac{x+\gamma}{1+\gamma};\, q\right)_{\!n-k-1} \\
&= \frac{1}{\left(\dfrac{\gamma}{1+\gamma};\, q\right)_{\!n}}
   \sum_{k=0}^{n-1}
   \left(\frac{1 + q[k]_q}{[n]_q}\right)
   \begin{bmatrix} n-1 \\ k \end{bmatrix}_{\!q}
   \left(\frac{x+\gamma}{1+\gamma}\right)^{\!k+1}
   \left(\frac{\gamma}{x+\gamma};\, q\right)_{\!k+1}
   \left(\frac{x+\gamma}{1+\gamma};\, q\right)_{\!n-k-1}
  \nonumber \\[10pt]
&= \frac{1}{[n]_q \left(\dfrac{\gamma}{1+\gamma};\, q\right)_{\!n}}
   \sum_{k=0}^{n-1}
   \begin{bmatrix} n-1 \\ k \end{bmatrix}_{\!q}
   \left(\frac{x+\gamma}{1+\gamma}\right)^{\!k+1}
   \left(\frac{\gamma}{x+\gamma};\, q\right)_{\!k+1}
   \left(\frac{x+\gamma}{1+\gamma};\, q\right)_{\!n-k-1}
   \nonumber \\[6pt]
&\quad + \frac{q}{[n]_q \left(\dfrac{\gamma}{1+\gamma};\, q\right)_{\!n}}
   \sum_{k=0}^{n-1}
   [k]_q
   \begin{bmatrix} n-1 \\ k \end{bmatrix}_{\!q}
   \left(\frac{x+\gamma}{1+\gamma}\right)^{\!k+1}
   \left(\frac{\gamma}{x+\gamma};\, q\right)_{\!k+1}
   \left(\frac{x+\gamma}{1+\gamma};\, q\right)_{\!n-k-1}
   \nonumber \\[10pt]
&=: I_1 + I_2. \nonumber
\end{align*}
Taking into account equality \eqref{30} for the sum 
$I_1$ and making the necessary arrangements,
\begin{align*}
I_1 &= \frac{\left(\dfrac{x+\gamma}{1+\gamma}\right)\!\left(1 - \dfrac{\gamma}{x+\gamma}\right)}
           {[n]_q \left(\dfrac{\gamma}{1+\gamma};\, q\right)_{\!n}}
      \sum_{k=0}^{n-1}
      \begin{bmatrix} n-1 \\ k \end{bmatrix}_{\!q}
      \left(\frac{x+\gamma}{1+\gamma}\right)^{\!k}
      \left(\frac{q\gamma}{x+\gamma};\, q\right)_{\!k}
      \left(\frac{x+\gamma}{1+\gamma};\, q\right)_{\!n-k-1} \nonumber \\
&= \frac{\left(\dfrac{x}{1+\gamma}\right)}{[n]_q \left(\dfrac{\gamma}{1+\gamma};\, q\right)_{\!n}}
  \left(\frac{q\gamma}{1+\gamma};\, q\right)_{\!n-1}
= \frac{x}{[n]_q(1+\gamma)\!\left(1 - \dfrac{\gamma}{1+\gamma}\right)}
= \frac{x}{[n]_q}
\end{align*}
is obtained. For the term $I_2$, one has
\begin{align*}
I_2 &= \frac{q}{[n]_q \left(\dfrac{\gamma}{1+\gamma};\, q\right)_{\!n}}
      \sum_{k=0}^{n-1}
      [k]_q
      \begin{bmatrix} n-1 \\ k \end{bmatrix}_{\!q}
      \left(\frac{x+\gamma}{1+\gamma}\right)^{\!k+1}
      \left(\frac{\gamma}{x+\gamma};\, q\right)_{\!k+1}
      \left(\frac{x+\gamma}{1+\gamma};\, q\right)_{\!n-k-1} \\
 &= \frac{q[n-1]_q}{[n]_q \left(\dfrac{\gamma}{1+\gamma};\, q\right)_{\!n}}
      \sum_{k=1}^{n-1}
      \begin{bmatrix} n-2 \\ k-1 \end{bmatrix}_{\!q}
      \left(\frac{x+\gamma}{1+\gamma}\right)^{\!k+1}
      \left(\frac{\gamma}{x+\gamma};\, q\right)_{\!k+1}
      \left(\frac{x+\gamma}{1+\gamma};\, q\right)_{\!n-k-1} \\
&= \frac{q[n-1]_q}{[n]_q \left(\dfrac{\gamma}{1+\gamma};\, q\right)_{\!n}}
      \sum_{k=0}^{n-2}
      \begin{bmatrix} n-2 \\ k \end{bmatrix}_{\!q}
      \left(\frac{x+\gamma}{1+\gamma}\right)^{\!k+2}
      \left(\frac{\gamma}{x+\gamma};\, q\right)_{\!k+2}
      \left(\frac{x+\gamma}{1+\gamma};\, q\right)_{\!n-k-2}.
\end{align*}
Applying \eqref{30} to the last expression, we arrive at
\begin{align*}
I_2 &= \frac{q[n-1]_q}{[n]_q \left(\dfrac{\gamma}{1+\gamma};\, q\right)_{\!n}}
       \left(\frac{x+\gamma}{1+\gamma}\right)^{\!2}
       \left(1 - \frac{\gamma}{x+\gamma}\right)
       \left(1 - q\frac{\gamma}{x+\gamma}\right)
       \notag \\[6pt]
&\quad \times
       \sum_{k=0}^{n-2}
       \begin{bmatrix} n-2 \\ k \end{bmatrix}_{\!q}
       \left(\frac{x+\gamma}{1+\gamma}\right)^{\!k}
       \left(\frac{q^2\gamma}{x+\gamma};\, q\right)_{\!k}
       \left(\frac{x+\gamma}{1+\gamma};\, q\right)_{\!n-k-2}
       \notag \\[10pt]
&= \frac{q[n-1]_q}{[n]_q \left(\dfrac{\gamma}{1+\gamma};\, q\right)_{\!n}}
   \left(\frac{x+\gamma}{1+\gamma}\right)^{\!2}
   \left(\frac{x}{x+\gamma}\right)
   \left(\frac{x+(1-q)\gamma}{x+\gamma}\right)
   \left(\frac{q^2\gamma}{1+\gamma};\, q\right)_{\!n-2} \\
&= \frac{q[n-1]_q}{[n]_q} \cdot \frac{x\bigl(x + (1-q)\gamma\bigr)}{1 + (1-q)\gamma}.
\end{align*}
Combining the results obtained for $I_1$ and $I_2$, the proof is complete.
\end{proof}
Up to this point in our work, the first three moments of the $q$-Stancu operators have been considered. We now turn our attention to the more general case of the $m$-th order moments of these operators, where the recurrence relations will be examined. These relations allow any $m$-th order moment (for $m = 0, 1, 2, \dots$) to be expressed in terms of the moments of lower order.

\begin{theorem}
Let $m \in \mathbb{N}_0$ and $\alpha \ge 0$. Then, for every $n \in \mathbb{N}$ and $x \in [0,1]$, the $q$-Stancu operators satisfy the recurrence relation
\begin{align}
S_n^{q,\alpha}(e_{m+1};x)
=
\frac{x}{[n]_q^{\,m}}
\sum_{s=0}^{m}
\binom{m}{s} q^{s}
[n-1]_q^{\,s}
S_{n-1}^{q,\frac{q\alpha}{1+\alpha}}
\left(e_s;\frac{x+\alpha}{1+\alpha}\right). \label{rec1}
\end{align}
\end{theorem}
\begin{proof}
Taking $f(t)=e_{m+1}(t)$ in \eqref{qstancu}, we obtain
\begin{align*}
S_n^{q,\alpha}(e_{m+1};x)
&=
\frac{1}{\left(\dfrac{\gamma}{1+\gamma};q\right)_n}
\sum_{k=0}^{n}
\left(\frac{[k]_q}{[n]_q}\right)^{m+1}
\qbinom{n}{k}{q}
\left(\frac{x+\gamma}{1+\gamma}\right)^k
\left(\frac{\gamma}{x+\gamma};q\right)_k
\left(\frac{x+\gamma}{1+\gamma};q\right)_{n-k} \\
&=
\frac{1}{\left(\dfrac{\gamma}{1+\gamma};q\right)_n}
\sum_{k=1}^{n}
\left(\frac{[k]_q}{[n]_q}\right)^{m}
\qbinom{n-1}{k-1}{q}
\left(\frac{x+\gamma}{1+\gamma}\right)^k
\left(\frac{\gamma}{x+\gamma};q\right)_k
\left(\frac{x+\gamma}{1+\gamma};q\right)_{n-k}\\
&=
\frac{1}{\left(\dfrac{\gamma}{1+\gamma};q\right)_n}
\sum_{k=0}^{n-1}
\left(\frac{[k+1]_q}{[n]_q}\right)^{m}
\qbinom{n-1}{k}{q}
\left(\frac{x+\gamma}{1+\gamma}\right)^{k+1}
\left(\frac{\gamma}{x+\gamma};q\right)_{k+1}
\left(\frac{x+\gamma}{1+\gamma};q\right)_{n-k-1}.
\end{align*}
Here, by using the identity $[k+1]_q = 1+q[k]_q$ together with the binomial expansion, we obtain
\begin{align*}
&S_n^{q,\alpha}(e_{m+1};x) \\
&=
\frac{1}{\left(\dfrac{\gamma}{1+\gamma};q\right)_n}
\sum_{k=0}^{n-1}
\left(\frac{1+q[k]_q}{[n]_q}\right)^{m}
\qbinom{n-1}{k}{q}
\left(\frac{x+\gamma}{1+\gamma}\right)^{k+1}
\left(\frac{\gamma}{x+\gamma};q\right)_{k+1}
\left(\frac{x+\gamma}{1+\gamma};q\right)_{n-k-1} \\
&= \frac{1}{\left(\dfrac{\gamma}{1+\gamma};\, q\right)_{\!n} [n]_q^m}
   \sum_{k=0}^{n-1} \sum_{s=0}^{m}
   \binom{m}{s} q^s [k]_q^s
   \begin{bmatrix} n-1 \\ k \end{bmatrix}_{\!q}
   \left(\frac{x+\gamma}{1+\gamma}\right)^{\!k+1}
   \left(\frac{\gamma}{x+\gamma};\, q\right)_{\!k+1}
   \left(\frac{x+\gamma}{1+\gamma};\, q\right)_{\!n-k-1}.
\end{align*}
After carrying out the necessary rearrangements in this expression, we find
\begin{align}
S_n^{q,\alpha}(e_{m+1};x)
&=
\frac{1}{\left(\dfrac{\gamma}{1+\gamma};q\right)_n [n]_q^{\,m}}
\sum_{s=0}^{m} \binom{m}{s} q^{s}
\left(\frac{x+\gamma}{1+\gamma}\right)
\left(1-\frac{\gamma}{x+\gamma}\right) \notag\\
&\quad\times
\Bigg\{
\sum_{k=0}^{n-1}
[k]_q^{\,s}
\qbinom{n-1}{k}{q}
\left(\frac{x+\gamma}{1+\gamma}\right)^k
\left(\frac{q\gamma}{x+\gamma};q\right)_k
\left(\frac{x+\gamma}{1+\gamma};q\right)_{n-k-1}
\Bigg\} \notag\\[2mm]
&= \frac{1}{\left(\dfrac{\gamma}{1+\gamma};\, q\right)_{\!n} [n]_q^m}
   \left(\frac{x}{1+\gamma}\right)
   \sum_{s=0}^{m}
   \binom{m}{s} q^s [n-1]_q^s
   \notag \\[6pt]
&\quad \times
   \left\{
   \sum_{k=0}^{n-1}
   \left(\frac{[k]_q}{[n-1]_q}\right)^{\!s}
   \begin{bmatrix} n-1 \\ k \end{bmatrix}_{\!q}
   \left(\frac{x+\gamma}{1+\gamma}\right)^{\!k}
   \left(\frac{q\gamma}{x+\gamma};\, q\right)_{\!k}
   \left(\frac{x+\gamma}{1+\gamma};\, q\right)_{\!n-k-1}
   \right\} \label{kk}
\end{align}
Let the second sum in the last equality be denoted by
\[
\Omega_s := \sum_{k=0}^{n-1}
\left(\frac{[k]_q}{[n-1]_q}\right)^s
\qbinom{n-1}{k}{q}
\left(\frac{x+\gamma}{1+\gamma}\right)^k
\left(\frac{q\gamma}{x+\gamma};q\right)_k
\left(\frac{x+\gamma}{1+\gamma};q\right)_{n-k-1}.
\]
Then, replacing $\alpha$ by $\dfrac{q\alpha}{1+\alpha}$ and $x$ by
$\dfrac{x+\alpha}{1+\alpha}$ in the new representation of the
$q$-Stancu operators, we get
\begin{align*}
&S_{n-1}^{\,q,\frac{q\alpha}{1+\alpha}} \left(e_s;\frac{x+\alpha}{1+\alpha}\right) \\
&=
\frac{1}{\left(\dfrac{q\gamma}{1+\gamma};q\right)_{n-1}}
\sum_{k=0}^{n-1}
\left(\frac{[k]_q}{[n-1]_q}\right)^s
\qbinom{n-1}{k}{q} 
\left(\frac{x+\gamma}{1+\gamma}\right)^k
\left(\frac{q\gamma}{x+\gamma};q\right)_k
\left(\frac{x+\gamma}{1+\gamma};q\right)_{n-k-1} \\
&=
\frac{1}{\left(\dfrac{q\gamma}{1+\gamma};q\right)_{n-1}}\,\Omega_s .
\end{align*}
Substituting the last equality into \eqref{kk}, we arrived at
\begin{align*}
&S_n^{q,\alpha}(e_{m+1};x) \\
&=
\frac{1}{\left(\dfrac{\gamma}{1+\gamma};q\right)_n [n]_q^{\,m}}
\left(\frac{x}{1+\gamma}\right)
\sum_{s=0}^{m} \binom{m}{s} q^{s} [n-1]_q^{\,s}
\left(\frac{q\gamma}{1+\gamma};q\right)_{n-1}
S_{n-1}^{\,q,\frac{q\alpha}{1+\alpha}}
\left(e_s;\frac{x+\alpha}{1+\alpha}\right) \\
&=
\frac{x}{[n]_q^{\,m}}
\sum_{s=0}^{m} \binom{m}{s} q^{s} [n-1]_q^{\,s}
S_{n-1}^{\,q,\frac{q\alpha}{1+\alpha}}
\left(e_s;\frac{x+\alpha}{1+\alpha}\right),
\end{align*}
which yields the desired recurrence relation.
\end{proof}
We now obtain an alternative recurrence relation for these operators by employing a different approach.
For this purpose, we follow the method used in the proof of relation in \cite{videnskii} for the $q$-Bernstein operators.

\begin{lemma}
Let $n \in \mathbb{N}$. Then, for $0 \le k \le n$, the following identity holds:
\begin{align}
\frac{[k]_q}{[n]_q}\, p_{n,k}^{q,\alpha}(x)
=
p_{n,k}^{q,\alpha}(x)
-
(1-x)\, p_{n-1,k}^{q,\frac{q\alpha}{1+\alpha}}
\!\left(\frac{q x}{1+\alpha}\right).
\label{basrec}
\end{align}
\end{lemma}
\begin{proof}
First, using the identity $[n]_q = [k]_q + q^k [n-k]_q$, we obtain
\begin{align*}
\frac{[k]_q}{[n]_q} p_{n,k}^{q,\alpha}(x)
&=
\left(\frac{[n]_q-q^k[n-k]_q}{[n]_q}\right)p_{n,k}^{q,\alpha}(x) \\
&=
p_{n,k}^{q,\alpha}(x)
-
\frac{q^k[n-k]_q}{[n]_q}\,p_{n,k}^{q,\alpha}(x).
\end{align*}

Substituting the basis polynomials given in \eqref{newbasis} into this expression yields
\begin{align}
\frac{[k]_q}{[n]_q} \, p_{n,k}^{q,\alpha}(x)
&=
p_{n,k}^{q,\alpha}(x)
-
\frac{q^{k}[n-k]_q}{[n]_q}
\qbinom{n}{k}{q}
\left(\frac{x+\gamma}{1+\gamma}\right)^k
\frac{
\left(\dfrac{\gamma}{x+\gamma};q\right)_k
\left(\dfrac{x+\gamma}{1+\gamma};q\right)_{n-k}
}{
\left(\dfrac{\gamma}{1+\gamma};q\right)_n
} \nonumber \\
&=
p_{n,k}^{q,\alpha}(x)
-
\qbinom{n-1}{k}{q}
\left(q \left(\dfrac{x+\gamma}{1+\gamma}\right) \right)^k
\frac{
\left(\dfrac{\gamma}{x+\gamma};q\right)_k
\left(\dfrac{x+\gamma}{1+\gamma};q\right)_{n-k}
}{
\left(\dfrac{\gamma}{1+\gamma};q\right)_n
} \nonumber \\
&=  
p_{n,k}^{q,\alpha}(x)
-
(1-x)\,
\qbinom{n-1}{k}{q}
\left(q\left(\dfrac{x+\gamma}{1+\gamma}\right) \right)^k
\dfrac{\left(\dfrac{\gamma}{x+\gamma};q\right)_k
\left(q\left(\dfrac{x+\gamma}{1+\gamma}\right);q \right)_{n-k-1}}{\left(\dfrac{q\gamma}{1+\gamma};q\right)_{n-1}}. \label{58}
\end{align}

Finally, replacing $\alpha$ by $q\alpha/(1+\alpha)$ and $x$ by $qx/(1+\alpha)$ in \eqref{newbasis}, we get
\begin{align*}
p_{n-1,k}^{q,\frac{q\alpha}{1+\alpha}}\!\left(\frac{qx}{1+\alpha}\right)
=
\qbinom{n-1}{k}{q}
\left(q\left(\dfrac{x+\gamma}{1+\gamma}\right) \right)^k
\frac{\left(\dfrac{\gamma}{x+\gamma};q\right)_k
\left(q\left(\dfrac{x+\gamma}{1+\gamma}\right);q\right)_{n-k-1}}{\left(\dfrac{q\gamma}{1+\gamma};q\right)_{n-1}}.
\end{align*}

Substituting this expression into \eqref{58} gives
\[
\frac{[k]_q}{[n]_q} p_{n,k}^{q,\alpha}(x)
=
p_{n,k}^{q,\alpha}(x)
-
(1-x)\,
p_{n-1,k}^{q,\frac{q\alpha}{1+\alpha}}\!\left(\frac{qx}{1+\alpha}\right),
\]
which completes the proof.
\end{proof}
\begin{theorem}
Let $m\in\mathbb{N}_0$ and $\alpha\ge 0$. Then, for every $n\in\mathbb{N}$ and
$x\in[0,1]$, the $q$-Stancu operators satisfy the recurrence relation
\begin{align}
S_n^{q,\alpha}(e_{m+1};x)
&=
S_n^{q,\alpha}(e_m;x)
-
(1-x)\left(\frac{[n-1]_q}{[n]_q}\right)^m
S_{n-1}^{\,q,\frac{q\alpha}{1+\alpha}}
\left(e_m;\frac{q x}{1+\alpha}\right).
\label{rec2}
\end{align}
\end{theorem}
\begin{proof}
Taking $\varphi(t)=e_{m+1}(t)$ in \eqref{qstancu}, we obtain
\begin{align*}
S_n^{q,\alpha}(e_{m+1};x)
&=
\sum_{k=0}^{n}
\left(\frac{[k]_q}{[n]_q}\right)^{m+1}
p_{n,k}^{q,\alpha}(x) =\sum_{k=0}^{n}\left(\frac{[k]_q}{[n]_q}\right)^{m}\frac{[k]_q}{[n]_q}p_{n,k}^{q,\alpha}(x).
\end{align*}

Using property \eqref{basrec} satisfied by the basis polynomials, we have
\begin{align*}
S_n^{q,\alpha}(e_{m+1};x)
&=
\sum_{k=0}^{n}
\left(\frac{[k]_q}{[n]_q}\right)^{m}
\left(
p_{n,k}^{q,\alpha}(x)
-
(1-x)\,
p_{n-1,k}^{\,q,\frac{q\alpha}{1+\alpha}}
\!\left(\frac{q x}{1+\alpha}\right)
\right) \\
&=
\sum_{k=0}^{n}
\left(\frac{[k]_q}{[n]_q}\right)^{m}
p_{n,k}^{q,\alpha}(x)
-
(1-x)
\sum_{k=0}^{n-1}
\left(\frac{[k]_q}{[n]_q}\right)^{m}
p_{n-1,k}^{\,q,\frac{q\alpha}{1+\alpha}}
\!\left(\frac{q x}{1+\alpha}\right).
\end{align*}

A direct calculation yields 
\begin{align*}
S_n^{q,\alpha}(e_{m+1};x)
&=
S_n^{q,\alpha}(e_m;x)
-
(1-x)
\left(\frac{[n-1]_q}{[n]_q}\right)^m
\sum_{k=0}^{n-1}
\left(\frac{[k]_q}{[n-1]_q}\right)^m
p_{n-1,k}^{\,q,\frac{q\alpha}{1+\alpha}}
\!\left(\frac{q x}{1+\alpha}\right) \\
&=
S_n^{q,\alpha}(e_m;x)
-
(1-x)
\left(\frac{[n-1]_q}{[n]_q}\right)^m
S_{n-1}^{\,q,\frac{q\alpha}{1+\alpha}}
\left(e_m;\frac{q x}{1+\alpha}\right).
\end{align*}
This completes the proof.
\end{proof}
It should be noted that, by taking $m=1,2$ in the recurrence relations \eqref{rec1} and \eqref{rec2}, one recovers the moments of the $q$-Stancu operators given in Theorem \ref{moment}. Moreover, setting $\alpha=0$ in \eqref{rec1} reduces the formula to the recurrence relation for the $q$-Bernstein polynomials in \cite{ostrovska2003}. Similarly, letting $\alpha=0$ in \eqref{rec2} yields the recurrence relation derived in \cite{videnskii}.

\section{On the Limit $q$-Stancu Operator}

The key step toward obtaining the limit operator lies in the transformation of the basis polynomials into the form \eqref{newbasis}. By expressing the basis polynomials in this form, as established in the preceding section, the infinite products appearing in the operators are shown to converge as $n \to \infty$. This enables an explicit definition of the limit $q$-Stancu operator. In this section, we introduce the limit form of the $q$-Stancu operator and investigate its basic properties.

Before constructing the desired limit operator, we present some auxiliary results, which are important in their own right and will also be needed in the subsequent steps.

\begin{lemma}
$\varphi \in C[0,1]$ and $\alpha \geq 0$. The series
\begin{align}
    \sum_{k=0}^{\infty} \frac{\varphi(1-q^k)}{(q;q)_k} \left(\frac{\gamma}{x+\gamma};q\right)_k \left(\frac{x+\gamma}{1+\gamma}\right)^k \label{conv1}
\end{align}
converges uniformly on every closed interval $[a,b] \subseteq [0,1)$.
\end{lemma}

\begin{proof}
Since $\varphi$ is continuous on $[0,1]$, it is also bounded. Thus there exists a positive constant $K$ such that $|\varphi(x)| \leq K$ for all $x \in [0,1]$.

Moreover, for every $k \in \mathbb{N}_0$ and $x \in [0,1]$, we have $(q;q)_k \geq (q;q)_\infty$ and $\left|\left(\frac{\gamma}{x+\gamma};q\right)_k\right| \leq 1$, so that
\begin{align*}
    \left|\frac{\varphi(1-q^k)\left(\frac{\gamma}{x+\gamma};q\right)_k}{(q;q)_k} \left(\frac{x+\gamma}{1+\gamma}\right)^k\right| \leq \frac{K}{(q;q)_\infty} \left|\frac{x+\gamma}{1+\gamma}\right|^k
\end{align*}
Now, since $[a,b]$ is an arbitrary closed subset of $[0,1)$,
\begin{align*}
    \left|\frac{x+\gamma}{1+\gamma}\right| \leq \frac{b+\gamma}{1+\gamma} < 1
\end{align*}
and therefore by the Weierstrass $M$-test, the series converges uniformly on $[a,b]$.
\end{proof}

\begin{corollary} \label{3.2}
Let $\varphi \in C[0,1]$ and $\alpha \ge 0$. Then, the series given by \eqref{conv1} converges for every $x \in [0,1)$.
\end{corollary}

\begin{lemma}
Let $k = 0,1,2,\dots$. Then,
\begin{align}
\lim_{n \to \infty} p_{nk}^{q,\alpha}(x) =
\dfrac{
\left( \dfrac{\gamma}{x+\gamma}; q \right)_k 
\left( \dfrac{x+\gamma}{1+\gamma}; q \right)_\infty 
\left( \dfrac{x+\gamma}{1+\gamma} \right)^k
}{
(q; q)_k \left( \dfrac{\gamma}{1+\gamma}; q \right)_\infty 
} 
=: p_{\infty k}^{q,\alpha}(x), \quad x \in [0,1], \label{conv2}
\end{align}
and the convergence is uniform.
\end{lemma}

\begin{proof}
Considering equality \eqref{binompoch}, it is seen that
\begin{align}
\lim_{n \to \infty} \begin{bmatrix} n \\ k \end{bmatrix}_q \dfrac{1}{\left(\dfrac{\gamma}{1+\gamma}; q\right)_n} = \dfrac{1}{(q;q)_k \left(\dfrac{\gamma}{1+\gamma}; q\right)_\infty} \label{conv3}
\end{align}
Additionally, one has
\begin{align*}
0 &\leq \left(\dfrac{x+\gamma}{1+\gamma}; q\right)_{n-k} - \left(\dfrac{x+\gamma}{1+\gamma}; q\right)_\infty= \left(\dfrac{x+\gamma}{1+\gamma}; q\right)_{n-k} \left[1 - \left(\left(\dfrac{x+\gamma}{1+\gamma}\right) q^{n-k}; q\right)_\infty\right] \notag \\[6pt]
&\leq 1 - \prod_{i=n-k}^{\infty} \left(1 - \left(\dfrac{x+\gamma}{1+\gamma}\right) q^i\right) \leq 1 - \prod_{i=n-k}^{\infty} \left(1 - q^i\right)
\end{align*}
At this point, the following Bernoulli inequality, well known in the literature, will be useful, for $a_1, a_2, \ldots, a_n \in (0,1)$
\begin{align}
1 - \prod_{i=1}^{n}(1 - a_i) \leq \sum_{i=1}^{n} a_i
\end{align}
holds. Taking the limit $n \to \infty$ on both sides of this inequality, for every sequence $\{a_i\}_{i=1}^{\infty}$ with terms in $(0,1)$
\begin{align*}
1 - \prod_{i=1}^{\infty}(1 - a_i) \leq \sum_{i=1}^{\infty} a_i
\end{align*}
is satisfied. Therefore, for every $x \in [0,1]$
\begin{align*}
0 \leq \left(\dfrac{x+\gamma}{1+\gamma}; q\right)_{n-k} - \left(\dfrac{x+\gamma}{1+\gamma}; q\right)_\infty
\leq 1 - \prod_{i=n-k}^{\infty}\left(1 - q^i\right)
\leq \sum_{i=n-k}^{\infty} q^i = \dfrac{q^{n-k}}{1-q} 
\end{align*}
is obtained. Since the right-hand side of the last inequality converges to $0$ independently of the variable $x$,
\begin{align}
\lim_{n \to \infty} \left(\dfrac{x+\gamma}{1+\gamma}; q\right)_{n-k} = \left(\dfrac{x+\gamma}{1+\gamma}; q\right)_\infty \label{conv4}
\end{align}
the convergence is uniform on $[0,1]$. Considering expressions \eqref{conv3} together with \eqref{conv2}, the desired convergence follows immediately.
\end{proof}

\begin{lemma}
For $\alpha \geq 0$, $k \in \mathbb{N}_0$ and $x \in [0,1)$, we have $p_{\infty k}^{q,\alpha}(x) \geq 0$ and
\begin{align}
\sum_{k=0}^{\infty} p_{\infty k}^{q,\alpha}(x) = 1, \qquad x \in [0,1). \label{uralim}
\end{align}
\end{lemma}

\begin{lemma} \label{3.5}
For $\varphi \in C[0,1]$ and $\alpha \geq 0$,
\begin{align*}
\lim_{x \to 1^-} \sum_{k=0}^{\infty} \varphi(1-q^k)\, p_{\infty k}^{q,\alpha}(x) = \varphi(1).
\end{align*}
\end{lemma}

\begin{proof}
Let $\epsilon > 0$ be given. Since the functions $p_{\infty k}^{q,\alpha}(x)$ are continuous and $p_{\infty k}^{q,\alpha}(1) = 0$, there exists a $\delta > 0$ such that for $1 - \delta < x < 1$, $\left|p_{\infty k}^{q,\alpha}(x)\right| < \epsilon$ holds. On the other hand, since $\varphi \in C[0,1]$, there exists a positive integer $M$ such that for $k > M$, $|\varphi(1-q^k) - \varphi(1)| < \epsilon$ can be written. Now, considering property \eqref{uralim}, for every $x \in (1-\delta, 1)$:
\begin{align*}
\Delta &:= \left|\sum_{k=0}^{\infty} \varphi(1-q^k)\, p_{\infty k}^{q,\alpha}(x) - \varphi(1)\right| \notag \\[6pt]
&= \left|\sum_{k=0}^{\infty} \varphi(1-q^k)\, p_{\infty k}^{q,\alpha}(x) - \sum_{k=0}^{\infty} \varphi(1)\, p_{\infty k}^{q,\alpha}(x)\right| \notag \\[6pt]
&\leq \sum_{k=0}^{\infty} |\varphi(1-q^k) - \varphi(1)|\, p_{\infty k}^{q,\alpha}(x) \notag \\[6pt]
&= \sum_{k=0}^{M} |\varphi(1-q^k) - \varphi(1)|\, p_{\infty k}^{q,\alpha}(x) + \sum_{k=M+1}^{\infty} |\varphi(1-q^k) - \varphi(1)|\, p_{\infty k}^{q,\alpha}(x) \notag \\[6pt]
&\leq \sum_{k=0}^{M} 2\|\varphi\|\, \epsilon + \sum_{k=M+1}^{\infty} \epsilon\, p_{\infty k}^{q,\alpha}(x) \leq \left[2\|\varphi\|(M+1) + 1\right]\epsilon
\end{align*}
is obtained, and the last expression shows that $\lim_{x \to 1^-} \Delta = 0$.
\end{proof}

\begin{definition}
Let $\varphi \in C[0,1]$ and $\alpha \geq 0$. The limit $q$-Stancu operators are defined as

\begin{align}
S_{\infty}^{q,\alpha}(\varphi; x) =
\begin{cases}
\displaystyle\sum_{k=0}^{\infty} \varphi(1 - q^k)\, p_{\infty k}^{q,\alpha}(x), & x \in [0,1) \\[10pt]
\varphi(1), & x = 1
\end{cases} \label{limitstan}
\end{align}

where $p_{\infty k}^{q,\alpha}(x)$ denotes the basis functions given by \eqref{conv2}.
\end{definition}

In view of Corollary \ref{3.2} and Lemma \ref{3.5}, the operator $S_{\infty}^{q,\alpha}$ is well-defined on $C[0,1]$. For $\alpha = 0$, the limit $q$-Stancu operators reduce to the limit $q$-Bernstein operators introduced by Il'inskii and Ostrovska \cite{Ilinskii}. These operators are linear, positive, and preserve the values of functions at the endpoints. Moreover, the operator defined in \eqref{limitstan} is the limit of the sequence of polynomials $S_n^{q,\alpha}(\varphi;x)$, which justifies calling it the limit $q$-Stancu operator.

\begin{theorem}
Let $\varphi \in C[0,1]$. Then, 
\begin{align}
\lim_{n \to \infty} S_n^{q,\alpha}(\varphi; x) = S_{\infty}^{q,\alpha}(\varphi; x)
\end{align}
converges uniformly on the closed interval $[0,1]$.
\end{theorem}

\begin{proof}
Since the operators satisfy the end-point interpolation property, it suffices to prove the hypothesis of the theorem for $x \in [0,1)$. Since $\varphi$ is  continuous, there holds

\begin{align}
\lim_{n \to \infty} \varphi\!\left(\frac{[k]_q}{[n]_q}\right)
= \varphi\!\left(\lim_{n \to \infty} \frac{[k]_q}{[n]_q}\right)
= \varphi\!\left(\lim_{n \to \infty} \frac{1 - q^k}{1 - q^n}\right)
= \varphi(1 - q^k). \label{endpoint}
\end{align}
Let $\epsilon > 0$ be arbitrary. For this value of $\epsilon$, one can find $a \in (0,1)$ such that $|\varphi(t) - \varphi(1)| < \epsilon/3$ for all $t \in [a, 1]$. Moreover, there exists a positive integer $M$ satisfying the inequality $1 - q^{M+1} \geq a$. Now, for $n > M$ and $x \in [0,1)$, define
\begin{align*}
\Delta := \left|S_n^{q,\alpha}(\varphi;\, x) - S_\infty^{q,\alpha}(\varphi;\, x)\right|.
\end{align*}
Taking properties \eqref{uralim} into account and applying the triangle inequality,
\begin{align*}
\Delta 
&= \left| \sum_{k=0}^{n} \left( \varphi\left(\frac{[k]_q}{[n]_q}\right) - \varphi(1) \right) p_{nk}^{q,\alpha}(x) 
- \sum_{k=0}^{\infty} \left( \varphi(1 - q^k) - \varphi(1) \right) p_{\infty k}^{q,\alpha}(x) \right| \\
&\leq \left| \sum_{k=0}^{M} \left( \varphi\left(\frac{[k]_q}{[n]_q}\right) - \varphi(1) \right) p_{nk}^{q,\alpha}(x) 
- \sum_{k=0}^{M} \left( \varphi(1 - q^k) - \varphi(1) \right) p_{\infty k}^{q,\alpha}(x) \right| \\
&\quad + \sum_{k=M+1}^{n} \left| \varphi\left(\frac{[k]_q}{[n]_q}\right) - \varphi(1) \right| p_{nk}^{q,\alpha}(x) 
+ \sum_{k=M+1}^{\infty} \left| \varphi(1 - q^k) - \varphi(1) \right| p_{\infty k}^{q,\alpha}(x) \\
&=: I_1 + I_2 + I_3.
\end{align*}
For each $k = 0,1,\dots,M$, it follows from \eqref{conv2} and \eqref{endpoint} that
\[
\lim_{n \to \infty} \left( \varphi\left(\frac{[k]_q}{[n]_q}\right) - \varphi(1) \right) p_{nk}^{q,\alpha}(x)
= \left( \varphi(1 - q^k) - \varphi(1) \right) p_{\infty k}^{q,\alpha}(x),
\]
and the convergence is uniform on $[0,1]$.

Since the sum $I_1$ contains only finitely many terms, for sufficiently large $n$, one can write $I_1 < \varepsilon/3$.

Moreover, since for $n > k$ we have $[k]_q/[n]_q \in (1 - q^k, 1)$, by the choice of $M$, for all $k > M$, the inequalities
\[
\left| \varphi\left(\frac{[k]_q}{[n]_q}\right) - \varphi(1) \right| < \varepsilon/3 
\quad \text{and} \quad 
\left| \varphi(1 - q^k) - \varphi(1) \right| < \varepsilon/3
\]
hold. 
In addition,
\begin{align*}
I_2 &< \frac{\varepsilon}{3} \sum_{k=M+1}^{n} p_{nk}^{q,\alpha}(x) < \frac{\varepsilon}{3}, \\
I_3 &< \frac{\varepsilon}{3} \sum_{k=M+1}^{\infty} p_{\infty k}^{q,\alpha}(x) < \frac{\varepsilon}{3}.
\end{align*}
Taking all three parts into account, we obtain
\begin{align*}
\Delta \leq I_1 + I_2 + I_3 < \varepsilon.
\end{align*}
Since $\varepsilon > 0$ is arbitrary, it follows that the sequence of operators
$\{S_n^{q,\alpha}\}$ converges uniformly to the limit operator 
$S_{\infty}^{q,\alpha}$.
\end{proof}

\begin{theorem} \label{limmoment}
Let $\alpha \geq 0$. Then, for all $x \in [0,1]$,
\begin{align*}
S_{\infty}^{q,\alpha}(e_0; x) &= 1, \\
S_{\infty}^{q,\alpha}(e_1; x) &= x, \\
S_{\infty}^{q,\alpha}(e_2; x) &= x - \frac{q x(1 - x)}{1 + \alpha}.
\end{align*}
\end{theorem}

\begin{proof}
\item[i)] From the representation \eqref{limitstan} of the limit $q$-Stancu operators, by the property \eqref{uralim}, one writes
\begin{align*}
    S_{\infty}^{q, \alpha}(e_{0} ; x) = \sum_{k=0}^{\infty} p_{\infty k}^{q, \alpha}(x) = 1.
\end{align*}

\item[ii)] When taking $\varphi(t) = e_{1}(t)$ in \eqref{limitstan}, we have
\begin{align*}
    S_{\infty}^{q, \alpha}(e_{1} ; x) &= \dfrac{\left(\dfrac{x+\gamma}{1+\gamma} ; q\right)_{\infty}}{\left(\dfrac{\gamma}{1+\gamma} ; q\right)_{\infty}} \sum_{k=0}^{\infty} (1-q^k) \dfrac{\left(\dfrac{\gamma}{x+\gamma} ; q\right)_{k} \left(\dfrac{x+\gamma}{1+\gamma}\right)^{k}}{(q ; q)_{k}} \\ 
 &= \frac{\left(\dfrac{x+\gamma}{1+\gamma} ; q\right)_{\infty}}{\left(\dfrac{\gamma}{1+\gamma} ; q\right)_{\infty}} \sum_{k=0}^{\infty} \dfrac{\left(\dfrac{\gamma}{x+\gamma} ; q\right)_{k} \left(\dfrac{x+\gamma}{1+\gamma}\right)^{k}}{(q ; q)_{k}} \nonumber \\
&- \dfrac{\left(\dfrac{x+\gamma}{1+\gamma} ; q\right)_{\infty}}{\left(\dfrac{\gamma}{1+\gamma} ; q\right)_{\infty}} \sum_{k=0}^{\infty} \dfrac{\left(\dfrac{\gamma}{x+\gamma} ; q\right)_{k} \left(q \dfrac{x+\gamma}{1+\gamma}\right)^{k}}{(q ; q)_{k}}
\end{align*}

By virtue of \eqref{31} and part (i) of Theorem \ref{limmoment}, it can be written as
\begin{align*}
S_{\infty}^{q, \alpha}(e_{1} ; x) &= S_{\infty}^{q, \alpha}(e_{0} ; x) - \dfrac{\left(\dfrac{x+\gamma}{1+\gamma} ; q\right)_{\infty}}{\left(\dfrac{\gamma}{1+\gamma} ; q\right)_{\infty}} \cdot \dfrac{\left(q \left(\dfrac{\gamma}{1+\gamma} \right) ; q\right)_{\infty}}{\left(q \left(\dfrac{x+\gamma}{1+\gamma} \right) ; q\right)_{\infty}} \\
&= 1 - \dfrac{\left(1 - \dfrac{x+\gamma}{1+\gamma}\right)}{\left(1 - \dfrac{\gamma}{1+\gamma}\right)} = 1 - (1 - x) = x
\end{align*}
\item[iii)] The proof for part (iii) proceeds along the same lines as the results established above.
\end{proof}

\begin{corollary}
Taking into account the moments provided in Theorem \ref{limmoment}, the moments of the limit $q$-Bernstein operator are obtained for $\alpha = 0$ \cite{Ilinskii}.
\end{corollary}

Thus far, the first three moments of the limit $q$-Stancu operator have been presented. In a general setting, all moments are expressed in the following theorem.

\begin{theorem} \label{3.11}
For $\alpha \geq 0$, the moments of the limit $q$-Stancu operator for every $x \in [0,1]$ are given by
\begin{equation}
S_{\infty}^{q, \alpha}(e_{m} ; x) = \sum_{s=0}^{m} \binom{m}{s} (-1)^s \dfrac{\left(\dfrac{x+\gamma}{1+\gamma} ; q\right)_{s}}{\left(\dfrac{\gamma}{1+\gamma} ; q\right)_{s}}, \quad m = 0,1, \dots \label{allmoment}
\end{equation}
\end{theorem}

\begin{proof}
By the definition of the limit $q$-Stancu operators, we can write
\begin{align*}
    S_{\infty}^{q, \alpha}(e_{m} ; x) = \dfrac{\left(\dfrac{x+\gamma}{1+\gamma} ; q\right)_{\infty}}{\left(\dfrac{\gamma}{1+\gamma} ; q\right)_{\infty}} \sum_{k=0}^{\infty} (1-q^k)^m \dfrac{\left(\dfrac{\gamma}{x+\gamma} ; q\right)_{k} \left(\dfrac{x+\gamma}{1+\gamma}\right)^{k}}{(q ; q)_{k}}.
\end{align*}
By employing the binomial expansion, it follows that
\begin{align*}
    S_{\infty}^{q, \alpha}(e_{m} ; x) &= \dfrac{\left(\dfrac{x+\gamma}{1+\gamma} ; q\right)_{\infty}}{\left(\dfrac{\gamma}{1+\gamma} ; q\right)_{\infty}} \sum_{k=0}^{\infty} \sum_{s=0}^{m} \binom{m}{s} (-1)^s q^{ks} \dfrac{\left(\dfrac{\gamma}{x+\gamma} ; q\right)_{k} \left(\dfrac{x+\gamma}{1+\gamma}\right)^{k}}{(q ; q)_{k}} \nonumber \\
    &= \dfrac{\left(\dfrac{x+\gamma}{1+\gamma} ; q\right)_{\infty}}{\left(\dfrac{\gamma}{1+\gamma} ; q\right)_{\infty}} \sum_{s=0}^{m} \binom{m}{s} (-1)^s \sum_{k=0}^{\infty} \dfrac{\left(\dfrac{\gamma}{x+\gamma} ; q\right)_{k} \left(q^s \left(\dfrac{x+\gamma}{1+\gamma} \right)\right)^{k}}{(q ; q)_{k}}.
\end{align*}
By taking advantage of the relation \eqref{31}, we arrive at the following expression
\begin{align*}
    S_{\infty}^{q, \alpha}(e_{m} ; x) &= \dfrac{\left(\dfrac{x+\gamma}{1+\gamma} ; q\right)_{\infty}}{\left(\dfrac{\gamma}{1+\gamma} ; q\right)_{\infty}} \sum_{s=0}^{m} \binom{m}{s} (-1)^s \dfrac{\left(q^s \left(\dfrac{\gamma}{1+\gamma} \right); q\right)_{\infty}}{\left(q^s \left( \dfrac{x+\gamma}{1+\gamma} \right) ; q\right)_{\infty}}= \sum_{s=0}^{m} \binom{m}{s} (-1)^s \dfrac{\left(\dfrac{x+\gamma}{1+\gamma} ; q\right)_{s}}{\left(\dfrac{\gamma}{1+\gamma} ; q\right)_{s}}.
\end{align*}
Consequently, the desired result is obtained.
\end{proof}

\begin{theorem}
For $\alpha \geq 0$, the following recurrence relation holds for every $x \in [0,1]$
\begin{align*}
S_{\infty}^{q, \alpha}(e_{m+1} ; x) = S_{\infty}^{q, \alpha}(e_{m} ; x) - (1-x) S_{\infty}^{q, \frac{q \alpha}{1+\alpha}}\left(e_{m} ; \dfrac{q x}{1+\alpha}\right).
\end{align*}
\end{theorem}

\begin{proof}
From the expression \eqref{limitstan}, it can be written that
\begin{align*}
    S_{\infty}^{q, \alpha}(e_{m+1} ; x) = \sum_{s=0}^{m+1} \binom{m+1}{s} (-1)^s \frac{\left(\dfrac{x+\gamma}{1+\gamma} ; q\right)_{s}}{\left(\dfrac{\gamma}{1+\gamma} ; q\right)_{s}}.
\end{align*}
With the help of Pascal's identity, it follows that
\begin{align*}
    S_{\infty}^{q, \alpha}(e_{m+1} ; x) &= \sum_{s=0}^{m+1} \left[ \binom{m}{s} + \binom{m}{s-1} \right] (-1)^s \frac{\left(\dfrac{x+\gamma}{1+\gamma} ; q\right)_{s}}{\left(\dfrac{\gamma}{1+\gamma} ; q\right)_{s}} \nonumber \\
    &= \sum_{s=0}^{m+1} \binom{m}{s} (-1)^s \frac{\left(\dfrac{x+\gamma}{1+\gamma} ; q\right)_{s}}{\left(\dfrac{\gamma}{1+\gamma} ; q\right)_{s}} + \sum_{s=0}^{m+1} \binom{m}{s-1} (-1)^s \frac{\left(\dfrac{x+\gamma}{1+\gamma} ; q\right)_{s}}{\left(\dfrac{\gamma}{1+\gamma} ; q\right)_{s}}
\end{align*}
Since $\binom{m}{m+1} = 0$ in the first sum and $\binom{m}{-1} = 0$ in the second sum, if we consider the first sum from $s=0$ to $s=m$ and shift the index of the second sum by starting from $s=1$ and replacing $s$ with $s+1$, it follows that
\begin{align*}
    S_{\infty}^{q, \alpha}(e_{m+1} ; x) = \sum_{s=0}^{m} \binom{m}{s} (-1)^s \frac{\left(\dfrac{x+\gamma}{1+\gamma} ; q\right)_{s}}{\left(\dfrac{\gamma}{1+\gamma} ; q\right)_{s}} + \sum_{s=0}^{m} \binom{m}{s} (-1)^{s+1} \frac{\left(\dfrac{x+\gamma}{1+\gamma} ; q\right)_{s+1}}{\left(\dfrac{\gamma}{1+\gamma} ; q\right)_{s+1}}
\end{align*}
Taking Theorem \ref{3.11} into account, this can be rearranged as
\begin{align}
    S_{\infty}^{q, \alpha}(e_{m+1} ; x) = S_{\infty}^{q, \alpha}(e_{m} ; x) - \frac{\left(1 - \dfrac{x+\gamma}{1+\gamma}\right)}{\left(1 - \dfrac{\gamma}{1+\gamma}\right)} \sum_{s=0}^{m} \binom{m}{s} (-1)^s \frac{\left(q \left(\dfrac{x+\gamma}{1+\gamma} \right) ; q\right)_{s}}{\left(q \left(\dfrac{\gamma}{1+\gamma} \right) ; q\right)_{s}} \label{68}
\end{align}
Finally, by replacing $\alpha$ with $q\alpha / (1+\alpha)$ and $x$ with $qx / (1+\alpha)$ in expression \eqref{allmoment}, 
\begin{align*}
    S_{\infty}^{q, \frac{q \alpha}{1+\alpha}}\left(e_{m} ; \frac{q x}{1+\alpha}\right) = \sum_{s=0}^{m} \binom{m}{s}(-1)^{s} \frac{\left(q\left(\dfrac{x+\gamma}{1+\gamma}\right) ; q\right)_{s}}{\left(q\left(\dfrac{\gamma}{1+\gamma}\right) ; q\right)_{s}}
\end{align*}
is obtained. By substituting this result into \eqref{68}, it yields the following recurrence relation
\begin{align*}
    S_{\infty}^{q, \alpha}(e_{m+1} ; x) = S_{\infty}^{q, \alpha}(e_{m} ; x) - (1-x) S_{\infty}^{q, \frac{q \alpha}{1+\alpha}}\left(e_{m} ; \frac{q x}{1+\alpha}\right).
\end{align*}
\end{proof}

\begin{corollary}
By taking the limit as $n \to \infty$ in the recurrence relation given by \eqref{rec1}, the limit $q$-Stancu operator satisfies the following equality:
\begin{align*}
    S_{\infty}^{q, \alpha}(e_{m+1} ; x) = x \sum_{s=0}^{m} \binom{m}{s} q^s (1-q)^{m-s} S_{\infty}^{q, \frac{q \alpha}{1+\alpha}}\left(e_{s} ; \dfrac{x+\alpha}{1+\alpha}\right)
\end{align*}
Furthermore, by setting $\alpha = 0$ in this expression, one arrives at the relation satisfied by the limit $q$-Bernstein operator in \cite{ostrovska2003}.
\end{corollary}

\section*{Conflict of Interest:}

The authors state that there is no conflict of interest.

\section*{Funding Information:}

The authors state that no funding involved.

\section*{Financial interests:} 

The authors declare they have no financial interests.


\begin{thebibliography}{99}

\bibitem{agratini2010}
O. Agratini,
{\it On a $q$-analogue of Stancu operators},
Cent. Eur. J. Math. {\bf 8} (2010), 191--198.

\bibitem{askey}
G. E. Andrews, R. Askey, R. Roy,
{\it Special functions},
Encyclopedia of Mathematics and Its Applications, Cambridge Univ. Press, Cambridge, 1999.

\bibitem{erencin2012}
A. Erençin, G. Başcanbaz-Tunca, F. Taşdelen,
{\it Kantorovich type $q$-Bernstein-Stancu operators},
Studia Univ. Babeş-Bolyai Math. {\bf 57} (2012), no. 1, 89--105.

\bibitem{finta2016a}
Z. Finta,
{\it Approximation by Stancu type $q$-operators},
Ann. Univ. Ferrara {\bf 62} (2016), 217--230.

\bibitem{gupta2016}
V. Gupta, T. M. Rassias, H. Sharma,
{\it $q$-Durrmeyer operators based on Pólya distribution},
J. Nonlinear Sci. Appl. {\bf 9} (2016), no. 4, 1497--1504.

\bibitem{Ilinskii}
A. Il'inskii, S. Ostrovska,
{\it Convergence of generalized Bernstein polynomials},
J. Approx. Theory {\bf 116} (2002), no. 1, 100--112.

\bibitem{jiang}
Y. Jiang, J. Li,
{\it The rate of convergence of $q$-Bernstein-Stancu polynomials},
Int. J. Wavelets Multiresolut. Inf. Process. {\bf 7} (2009), no. 6, 773--779.

\bibitem{kac}
V. Kac, P. Cheung,
{\it Quantum calculus},
Springer, New York, 2001.

\bibitem{lupas}
A. Lupaş,
{\it A $q$-analogue of the Bernstein operator},
Seminar on Numerical and Statistical Calculus, Univ. Cluj-Napoca, No. 9 (1987).

\bibitem{mahmudov}
N. I. Mahmudov,
{\it Higher order limit $q$-Bernstein operators},
Math. Methods Appl. Sci. {\bf 34} (2011), no. 13, 1618--1626.

\bibitem{neer2017}
T. Neer, A. M. Acu, P. Agrawal,
{\it Approximation of functions by bivariate $q$-Stancu-Durrmeyer type operators},
Math. Commun. {\bf 23} (2018), no. 2, 161--180.

\bibitem{nowak}
G. Nowak,
{\it Approximation properties for generalized $q$-Bernstein polynomials},
J. Math. Anal. Appl. {\bf 350} (2009), no. 1, 50--55.

\bibitem{nowak_gupta2011}
G. Nowak, V. Gupta,
{\it The rate of pointwise approximation of positive linear operators based on $q$-integers},
Ukr. Math. J. {\bf 63} (2011), 403--415.

\bibitem{oruc}
H. Oruç, G. M. Phillips,
{\it A generalization of Bernstein polynomials},
Proc. Edinburgh Math. Soc. (2) {\bf 42} (1999), no. 2, 403--413.

\bibitem{ostrovska2003}
S. Ostrovska,
{\it $q$-Bernstein polynomials and their iterates},
J. Approx. Theory {\bf 123} (2003), no. 2, 232--255.

\bibitem{ostrovska2006}
S. Ostrovska,
{\it On the Lupaş $q$-analogue of the Bernstein operator},
Rocky Mountain J. Math. {\bf 36} (2006), no. 5, 1615--1629.

\bibitem{ostrovska2014}
S. Ostrovska, M. Turan,
{\it On the eigenvectors of the $q$-Bernstein operators},
Math. Methods Appl. Sci. {\bf 37} (2014), no. 4, 562--570.

\bibitem{ostrovska2023}
S. Ostrovska, M. Turan,
{\it On the block functions generating the limit $q$-Lupa\c{s} operator},
Quaest. Math. {\bf 46} (2023), no. 4, 711--719.

\bibitem{phillips}
G. M. Phillips,
{\it Bernstein polynomials based on the $q$-integers},
Ann. Numer. Math. {\bf 4} (1997), 511--518.

\bibitem{turan}
M. Turan,
{\it The truncated-Bernstein polynomials in the case $q>1$},
Abstr. Appl. Anal. 2014, Art. ID 126319.

\bibitem{videnskii}
V. S. Videnskii,
{\it On some classes of $q$-parametric positive linear operators},
Oper. Theory Adv. Appl. {\bf 158} (2005), 213--222.

\bibitem{wang2013}
Y. Wang, Y. Zhou,
{\it Iterates properties for $q$-Bernstein-Stancu operators},
Int. J. Model. Optim. {\bf 3} (2013), no. 4, 362--368.

\bibitem{wang2014}
Y. Wang, Y. Zhou,
{\it Shape preserving properties for $q$-Bernstein-Stancu operators},
J. Math. {\bf 2014} (2014), Art. ID 603694.

\end{thebibliography}
\end{document}